\documentclass[11pt]{amsart}

\usepackage{latexsym}
\usepackage{amssymb}
\usepackage{amsmath}
\usepackage{amsfonts}
\usepackage[mathscr]{eucal}
\usepackage{auto-pst-pdf}
\usepackage{times}
\usepackage{MnSymbol}
\usepackage{color}

\usepackage{amscd}

\usepackage{graphicx}
\usepackage{epsfig}

\usepackage{amsthm}
\usepackage[all,cmtip]{xy}
\usepackage{hyperref}
\usepackage{pst-grad} 
\usepackage{pst-plot} 
\usepackage{comment}
\usepackage{psfrag}
\usepackage{verbatim}

\newtheorem{thm}{Theorem}[section]

\newtheorem{lem}[thm]{Lemma}
\newtheorem{prop}[thm]{Proposition}

\newtheorem{theorem}{Theorem}

\theoremstyle{definition}

\newtheorem{example}[thm]{Example}
\newtheorem*{ack}{Acknowledgements}

\numberwithin{equation}{section}

\DeclareMathOperator{\Susp}{Susp}
\DeclareMathOperator{\curv}{curv}

\newcommand{\RP}{\mathbb{R}P}
\newcommand{\CP}{\mathbb{C}P}

\newcommand{\sphere}{\mathrm{\mathbb{S}}}

\newcommand{\SO}{\mathrm{SO}}

\newcommand{\Ss}{\mathrm{S}}
\newcommand{\T}{\mathrm{T}}
\newcommand{\ZZ}{\mathbb{Z}}
\newcommand{\bq}{/ \hspace{-.1cm} /}

\begin{document}



\title[Alexandrov $4$-manifolds with torus actions]{Simply connected Alexandrov $4$-manifolds with positive or nonnegative curvature and torus actions\\  
}

\author[F.\ Galaz-Garcia]{Fernando Galaz-Garcia$^\ast$}

\thanks{$^\ast$ The author is part of SFB 878: \emph{Groups, Geometry \& Actions}, at the University of M\"unster.}

\address{Mathematisches Institut, WWU M\"unster, Germany}
\email{f.galaz-garcia@uni-muenster.de}

\date{\today}


\subjclass[2000]{53C20, 57S15, 57S25}
\keywords{$4$-manifold, circle action, torus action, Alexandrov space}


\begin{abstract}
We point out that a $4$-dimensional topological manifold with an Alexandrov metric (of curvature bounded below) and with an effective, isometric action of the circle or the $2$-torus is locally smooth. This observation implies that the topological and equivariant classifications of compact, simply connected Riemannian $4$-manifolds with positive or nonnegative sectional curvature and an effective isometric action of a circle or a $2$-torus  also hold if we consider  Alexandrov manifolds instead of Riemannian manifolds. 
\end{abstract}
\maketitle




\section{Introduction and results}

The interaction between  (sectional) curvature and  isometric actions of compact Lie groups on Riemannian manifolds has played an important role in the study of compact Riemannian manifolds with positive or nonnegative curvature.  Alexandrov spaces (with curvature bounded below) arise naturally in this context as orbit spaces of isometric actions of compact Lie groups on Riemannian manifolds with a lower curvature bound. Alexandrov spaces also occur as Gromov-Hausdorff limits of sequences of Riemannian manifolds. Furthermore,  being synthetic generalizations of Riemannian manifolds with curvature bounded below, Alexandrov spaces are geometric objects of intrinsic interest.  

It is well-known that the isometry group of a (compact) Riemannian manifold is a (compact) Lie group (see, for example, \cite{Ko2}). Since the same conclusion is true for Alexandrov spaces (cf.~\cite{FY}), it is of interest  to determine which results in equivariant Riemannian geometry have Alexandrov analogues (cf.~\cite{GGG,GGS_C1A,HaSe}). In this note we show that the topological and equivariant classifications of compact, simply connected Riemannian $4$-manifolds with positive or nonnegative curvature and an effective isometric action of a circle or a $2$-torus (cf.~\cite{GG,GGK, GGS, GW, HK, Kl, SY}) still holds if instead of considering  Riemannian manifolds we consider \emph{Alexandrov manifolds}, i.e.~ topological manifolds equipped with an Alexandrov metric. 

Before stating our results, recall that a \emph{biquotient} is a quotient of a Lie group $G$ by the two-sided, free action of a subgroup $H \subset G \times G$.  If $G$ is equipped with a bi-invariant metric, then the action of $H$ is by isometries and the quotient $G \bq H$ equipped with the induced metric (of nonnegative curvature) is called a \emph{normal biquotient}. In the statements of the theorems we assume, without loss of generality, that the positively curved Alexandrov spaces have curvature bounded below by $1$. 


\begin{theorem}
\label{T:T2_CLAS}
Let $X$ be a compact, simply connected $4$-dimensional Alexandrov manifold with an effective, isometric $\T^2$ action. 
	\begin{enumerate}
		\item If $\curv(X)\geq 1$, then $X$ is equivariantly homeomorphic to $\sphere^4$ or $\CP^2$ with a linear action. \\
		
		\item If $\curv(X)\geq 0$, then $X$ is equivariantly homeomorphic to $\sphere^2$ or $\CP^2$ with a linear action or to $\sphere^2\times\sphere^2$ or $\CP^2\#\pm\CP^2$ with an action induced by a normal biquotient Riemannian metric. 
	\end{enumerate}
\end{theorem}


\begin{theorem}
\label{T:T1_CLAS}
Let $X$ be a compact, simply connected $4$-dimensional Alexandrov manifold  with an effective, isometric $\Ss^1$ action. 
	\begin{enumerate}
		\item If $\curv(X)\geq 1$, then $X$ is equivariantly homeomorphic to $\sphere^4$ or $\CP^2$ with a linear action. \\
		
		\item If $\curv(X)\geq 0$, then $X$ is equivariantly homeomorphic to $\sphere^2$ or $\CP^2$ with a linear action or to $\sphere^2\times\sphere^2$ or $\CP^2\#\pm\CP^2$ with an action induced by a normal biquotient Riemannian metric. 
	\end{enumerate}
\end{theorem}

Theorem~\ref{T:T2_CLAS} is a consequence of the work of Orlik and Raymond on compact, simply connected, topological $4$-manifolds with topological $2$-torus actions. Theorem~\ref{T:T1_CLAS} follows from the observation that the circle actions under consideration are locally smooth (see Subsection~\ref{SS:LOC_SM_ACTIONS}). This  allows us to use the work of Fintushel and Pao on compact, simply connected, topological $4$-manifolds with locally smooth circle actions \cite{F1,F2,Pa}. An application of metric and topological tools along the lines of the Riemannian case, including  recent work of Grove and Wilking on extremal knots in an Alexandrov $3$-sphere \cite{GW}, yields the result. 

The assumption in Theorem~\ref{T:T1_CLAS} that the circle acts by isometries of an Alexandrov metric is not superfluous, as there exist  topological circle actions on $\sphere^4$ which are not equivalent to smooth actions (cf.~\cite{MZ}). Recall that $\T^2$ is the largest torus that can act isometrically on a compact, simply connected $4$-dimensional Alexandrov space (cf.~\cite{GGS_C1A}). 
\\


To put our results into context, let us make a few remarks on Alexandrov manifolds and isometric actions on them. We will consider only manifolds without boundary, except when explicitly stated otherwise.  We recall first the fact  that in an Alexandrov manifold not every space of directions must be homeomorphic to a sphere. This phenomenon is illustrated by the following example (cf.~\cite{Ka}). 


\begin{example} 
\label{E:Alex_mfd}
Let $(P^3,g_0)$ be the Poincar\'e homology $3$-sphere with its standard Riemannian metric $g_0$ of constant sectional curvature $1$. The double spherical suspension $\Susp^2(P^3)$ is  a positively curved Alexandrov manifold and, by work of Edwards, it is homeomorphic to the $5$-sphere (see \cite{Da}). On the other hand, there are two points in $\Susp^2(P^3)$ with space of directions homeomorphic to $\Susp(P^3)$, which is not homeomorphic to the $4$-sphere. 
\end{example}

In contrast to the situation in dimension $5$,  every space of directions of an Alexandrov $n$-manifold, for $n\leq 4$, must be homeomorphic to a sphere. Indeed, in dimensions $1$, $2$ and $3$, the only possible spaces of directions are homeomorphic to $\sphere^0$ in dimension $1$,  to $\sphere^1$ in dimension $2$, and to $\sphere^2$ or $\RP^2$ in dimension $3$. In dimension $4$  any space of directions $\Sigma_xX$ of an Alexandrov $4$-manifold $X$ must  be a $3$-dimensional homology manifold and is therefore a topological $3$-manifold. Moreover, $\Sigma_xX$ must be simply connected (see, for example, \cite{Wu}) and, by Perelman's proof of the Poincar\'e conjecture, it must be homeomorphic to the $3$-sphere.

Recall that an (effective) isometric action of a compact Lie group $G$ on a Riemannian manifold $M$ with nonempty fixed point set is \emph{fixed point homogeneous} if some fixed point set component $F^*$ has codimension one in the orbit space $M^*$ of the action. In particular, $F^*$ must be a boundary component of $M^*$. By work of Grove and Searle  \cite{GS2}, any isometric fixed point homogeneous $G$-action on an $n$-sphere $\sphere^n$ equipped with a Riemannian metric of positive sectional curvature must be equivariantly diffeomorphic to an orthogonal action.  In this case, the action has exactly one fixed point set component $F^*$, and $F^*$ is diffeomorphic to a sphere. If we only assume that $\sphere^n$ carries an Alexandrov metric of positive curvature, this result is still true if $n\leq 4$, as a consequence of Theorem~A and the fact that topological actions of compact Lie groups on $\sphere^n$, for $n\leq 4$, with orbit spaces of dimension $1$ or $2$ are equivalent to orthogonal actions (see \cite{Ri} and references therein). On the other hand, for each $n\geq 5$, we can use the construction in Example~\ref{E:Alex_mfd} to create fixed point homogeneous isometric actions on $\sphere^n$ with an Alexandrov metric  of positive curvature whose fixed point set is the Poincar\'e $3$-sphere. Therefore, these actions cannot be equivalent to orthogonal actions, in contrast to the Riemannian case. 


\begin{example} Let $P^3$ be the Poincar\'e sphere with its Riemannian metric of constant sectional curvature $1$, and let  $\sphere^n$, $n\geq 1$, be the round $n$-sphere of constant sectional curvature $1$. Since the join $P^3*\sphere^1$ is homeomorphic to $\Susp^2(P^3)\simeq \sphere^5$, the spherical join $(P^3*\sphere^n,d)$ is an Alexandrov manifold homeomorphic to $\sphere^{n+4}$ and has curvature bounded below by $1$. Letting $\SO(n+1)$ act transitively on $\sphere^n$ and trivially on $P^3$, we obtain an effective isometric fixed point homogeneous $\SO(n+1)$-action on $\sphere^{n+4}$, $n\geq 1$, equipped with an Alexandrov metric of positive curvature. The fixed point set of this action is $P^3$. 
\end{example}

This note is divided as follows. In Section~\ref{S:Prelim} we  summarize some background material on isometric  actions of compact Lie groups on Alexandrov spaces (cf.~\cite{GGG,GGS_C1A,HaSe}) and show that the spaces considered in Theorems~A and B are locally smooth. Sections~\ref{S:PROOF_THM_A} and \ref{S:PROOF_THM_B} contain, respectively, the proofs of Theorems~A and B. 

We assume that the reader is acquainted with the basic theory of compact transformation groups (cf.~\cite{Br}) and of Alexandrov spaces (cf.~\cite{BBI,BGP}). The reader interested in more general aspects of group actions on topological $4$-manifolds may consult the survey \cite{Ed}.


\begin{ack} I wish to thank A.~Lytchak for constructive criticism on a preliminary version of this note. I would also like to thank  B.~Wilking, J.~Harvey and C.~Searle for conversations on their work contained in \cite{GW} and \cite{HaSe}.
\end{ack}


\section{Preliminaries}
\label{S:Prelim}


\subsection{Equivariant Alexandrov geometry}

Let $G\times X \longrightarrow X$,  $x \mapsto g(x)$, be a topological action of a topological group $G$ on a Hausdorff topological space $X$, i.e. let $X$ be a \emph{$G$-space}. We will denote the orbit of a point $x\in X$ under the action of $G$ by $G(x)\simeq G/G_x$, where $G_x = \{\, g \in G\, :\, g(x) = x \,\}$ is the isotropy subgroup of $G$ at $x$. Given a subset $A \subset X$, we will denote its projection under the orbit map $\pi: X \rightarrow X/G$ by $A^*$. Following this convention, we will denote the orbit space $X/G$  by $X^*$. We will henceforth assume all actions to be effective. We will say that two $G$-spaces are \emph{equivalent} if they are equivariantly homeomorphic.

Let $(X,d)$ be a (finite dimensional) Alexandrov space (with curvature bounded below) with an isometric action of a compact Lie group $G$. We will denote the space of directions at a point $x$ in $X$ by $\Sigma_x X$. Given $A\subset \Sigma_xX$, the set of \emph{normal directions} to $A$, denoted by $A^\perp$, is defined by
\[
A^\perp=\{\, v\in \Sigma_xX : d(v,w)=\mathrm{diam}(\Sigma_xX)/2\text{ for all } w\in A\,\}.
\]
Let $S_x\subset \Sigma_xX$ be  the unit tangent space to the orbit $G(x)$ at $x$ and suppose that  $\dim G(x)>0$. The set  $S_x^\perp$
 is a compact, totally geodesic Alexandrov subspace of $\Sigma_xX$ with curvature bounded below by $1$,  and the space of directions $\Sigma_x X$ is isometric to the join 
 $S_x * S_x^\perp$
  with the standard join metric. Moreover, either 
  $S_x^\perp$ is connected or it contains exactly two points at distance $\pi$ (cf.~\cite{GGS_C1A}).

 The slice and principal orbit theorems, as well as Kleiner's isotropy lemma for isometric actions of compact Lie groups on Riemannian  manifolds (cf.~\cite{Gr}), also hold for Alexandrov spaces (cf.~\cite{GGG,HaSe}). In particular, if a compact Lie group  $G$ acts by isometries on an Alexandrov space $X$, then a slice at $x$ is $G_x$-equivariantly homeomorphic to the cone on
 $S_x^\perp$,
  the space of normal directions to the orbit. In other words,  for any $x\in X$ there is some $r_0 > 0$ such that for all $r < r_0$ there is an equivariant homeomorphism   $\varphi : G \times_{G_x} K(S_x^\perp)\rightarrow  B_r(G(x))$
   (cf.~\cite{HaSe}). Hence $\Sigma_{x^*}X^*$, the space of directions of the orbit space $X^*$ at $x^*$, is isometric to    $S_x^\perp/G_x$.


\subsection{Locally smooth actions}
\label{SS:LOC_SM_ACTIONS}


Let $G$ be a compact Lie group and let $X$ be a $G$-space. Let $G(x)$ be an orbit of type $G/H$, for some $x\in X$. Recall that a \emph{tube} about $G(x)$ is a $G$-equivariant embedding (homeomorphism into) 
\[
\varphi: G\times_H A\longrightarrow X,
\]
onto an open neighborhood of $G(X)$ in $X$, where $A$ is some space on which $H$ acts. If $A$ is homeomorphic to Euclidean space $V$ and $H$ acts orthogonally on $V$, then we say that the tube 
\[
\varphi: G\times_H V\longrightarrow X
\]
is a \emph{linear tube}. A slice $S$ at $x\in X$ is called a \emph{linear slice} if the canonically associated tube
\[
G\times_{G_x}S\longrightarrow X
\]
\[
[g,s]\mapsto g(s)
\]
is equivalent to a linear tube, i.e.~if the $G_x$-space $S$ is equivalent to an orthogonal $G_x$-space. We say that a $G$-space is \emph{locally smooth} if there exists a linear tube about each orbit.  Since $G\times_H V$ is a $V$-bundle over $G/H$, it is a topological manifold. Thus, any  locally smooth $G$-space must be a topological manifold. If the $G$-space $X$ is locally smooth and if $x$ is a fixed point, then a neighborhood of $x$ in $X$ is equivalent to an orthogonal action. Hence, in this situation,  the fixed point set of the $G$ action is a topological submanifold of $X$. There exist examples of topological actions on manifolds which are not locally smooth (cf.~\cite[Ch. IV]{Br}).  We now show that the $G$-spaces considered in Theorems~A and B  are locally smooth.



\begin{lem}
\label{T:LOC_SM}
A $4$-dimensional Alexandrov manifold with an effective, isometric action of $\T^2$ or $\Ss^1$ is locally smooth.  
\end{lem}


\begin{proof}  
By the slice theorem, it suffices to show that the isotropy action on the normal space of directions to any orbit is equivalent to a linear action on a sphere.
 
Suppose first that we have an action of $\T^2$. In this case, the only possible non-principal isotropy subgroups of the torus action are conjugate to $\Ss^1$ or $\T^2$ (cf.~\cite{OR}). The space of directions normal  to the orbit is, respectively, homeomorphic to $\sphere^2$ or to $\sphere^3$ and it is well-known that topological $\Ss^1$ actions on $\sphere^2$ and topological $\T^2$ actions on $\sphere^3$ are equivalent to linear actions.

Suppose now that we have an $\Ss^1$ action. In this case, the only possible non-principal isotropy subgroups are finite cyclic groups $\ZZ_k$, for $k\geq 2$, or $\Ss^1$. The space of directions normal  to the orbit is, respectively, homeomorphic to $\sphere^2$ or to $\sphere^3$. In the first case, $\ZZ_k$ acts by homeomorphisms on $\sphere^2$. By work of de Ker\'ekj\'art\'o \cite{dK}, Brouwer \cite{Bro} and Eilenberg \cite{Ei},   this action is equivalent to an orthogonal action (cf.~\cite{CK}). In the second case,  $\Ss^1$ acts by homeomorphisms on $\sphere^3$ and, by work of Raymond \cite[Theorem~6]{Ra}, this action is equivalent to an orthogonal action.
\end{proof}


\section{Proof of Theorem~A}
 \label{S:PROOF_THM_A}
By work of Orlik and Raymond \cite{OR} , $X$ is homeomorphic to a connected sum of copies of $\sphere^2\times \sphere^2$, $\pm \CP^2$ and $\sphere^4$. Moreover,  the $\T^2$ action is equivalent to a smooth $\T^2$ action on $X$, where $X$ has the standard smooth structure induced via connected sum by the standard smooth structure on each one of $\sphere^2\times \sphere^2$, $\CP^2$ and $\sphere^4$. The orbit space $X^*$ is homeomorphic to a $2$-disk and the action has at least $2$ fixed points. On the other hand, by a comparison argument as in \cite{GGS}, we conclude that the action has at most $3$ fixed points, if $\curv(X)\geq 1$, and at most $4$ fixed points, if $\curv(X)\geq 0$. It follows from \cite{OR} that, if the action has $2$ or $3$ fixed points, then $X$ is equivariantly homeomorphic, respectively,  to $\sphere^4$ or $\CP^2$ equipped with a linear action; if the action has $4$ fixed points, then $X$ is equivariantly homeomorphic to $\sphere^2\times \sphere^2$ or $\CP^2\#\pm \CP^2$ with a smooth $\T^2$ action. By \cite{GGK}, any smooth $\T^2$ action on $\sphere^2\times \sphere^2$ or $\CP^2\#\pm \CP^2$ is equivariantly diffeomorphic to an isometric action on a normal biquotient.


\section{Proof of Theorem~B}
 \label{S:PROOF_THM_B}

\subsection{Initial setup}
\label{SS:SETUP} 
Let $X$ be a compact, simply connected $4$-dimensional Alexandrov manifold  with an isometric $\Ss^1$ action. 
By Lemma~\ref{T:LOC_SM}, the $\Ss^1$ action on $X$ is locally smooth. Fintushel analyzed in \cite{F1} locally smooth circle actions on compact, simply connected topological $4$-manifolds. In the next paragraph we recall those results in \cite{F1} that will be relevant to our discussion. 


The only possible orbit types of the action are principal orbits, fixed points and exceptional orbits. Their corresponding isotropy groups are, respectively, the trivial group, the group $\Ss^1$, and a finite cyclic group $\ZZ_k$,  $k\geq 2$. We will denote the set of fixed points by $F$ and the set of exceptional orbits by $E$. The orbit space $X^*$ is a simply connected topological $3$-manifold with boundary $\partial X^*\subset F^*$, the set $F^*-\partial X^*$ of isolated fixed points is finite and $F^*$ is nonempty. The boundary components of $X^*$ are $2$-spheres and the closure $\overline{E}^*$ of $E^*$ is a collection of polyhedral arcs and simple closed curves in $X^*$. The components of $E^*$ are open arcs on which orbit types are constant, and these arcs have closures with distinct endpoints in $F^*-\partial X^*$, so that $\overline{E}^*\subset E^*\cup F^*$. It follows from \cite[Theorem 7.1]{F1} and the validity of the Poincar\'e conjecture that, if   $\overline{E}^*$ contains no simple closed curves, then the 
$\Ss^1$ action on $X$ extends to 
a locally smooth action of $\T^2$. If $\overline{E}^*$ contains a simple closed curve $K^*$, then the $\Ss^1$ action extends to a locally smooth $\T^2$ action if and only if $E^*\cup F^*=K^*$ and $K^*$ is unknotted in $X^*\simeq \sphere^3$. The simple closed curves in $\overline{E}^*$ are tame knots in $X^*$.


The set $E^*\cup F^*\subset X^*$ of non principal orbits is an \emph{extremal set} in $X^*$, i.e.~given any point $p^*\in X^*- E^*\cup F^*$, any point $q^*\in  E^*\cup F^*$ with $d(p^*,q^*) = d(p^*,E^*\cup F^*)$ is a critical point of the distance function $d(p^*, \cdot)$; equivalently,  the space of directions $\Sigma_{q^*}X^*$ has diameter at most $\pi/2$. In particular, the space of directions at an isolated fixed point $p^*$ in $X^*$ is isometric to a $2$-sphere with an Alexandrov metric of positive curvature and is not larger than $\sphere^2(1/2)$, the round $2$-sphere with radius $1/2$, in the sense that there exists a distance decreasing map to the smaller space. Recall that that an extremal set that contains no proper extremal set of the same dimension is said to be \emph{primitive}. The $2$-dimensional components of $F^*$, which correspond to the components of $\partial X^*$, are primitive extremal sets in $X^*$. 


By work of Fintushel \cite{F1,F2}, Pao \cite{Pa} and Perelman's proof of the Poincar\'e conjecture, $X$ is equivariantly homeomorphic to a connected sum of copies of $\sphere^2\times \sphere^2$, $\pm \CP^2$ and $\sphere^4$ (\cite[Theorem 13.2]{F2}). The action is determined up to (orientation preserving) equivariant homeomorphism by Fintushel's  orbit space weights (see \cite[Section~3]{F1}). Moreover,  the $\Ss^1$ action is equivalent to a smooth $\Ss^1$ action on $X$ equipped with the standard smooth structure induced via connected sum by the standard smooth structure on each one of $\sphere^2\times \sphere^2$, $\CP^2$ and $\sphere^4$. Therefore, by a theorem of Kobayashi for smooth circle actions on compact smooth manifolds \cite[Theorem~5.5-(1)]{Ko2}, the Euler characteristic of $X$, which we denote by $\chi(X)$, is equal to $\chi(F)$, the Euler characteristic of the fixed point set of the circle action. Observe that $\chi(X)\geq 2$, as a consequence of Poincar\'e duality. Thus, to obtain the topological classification in Theorem~B it suffices to show that $\chi(F)=\chi(X)\leq 3$, when $X$ has positive curvature, and $\chi(F)=\chi(X)\leq 4$, when $X$ has nonnegative curvature.

We will use the following lemma, which follows from triangle comparison arguments (see, for example, \cite{GW}). 


\begin{lem}
\label{L:FP_Bound}
Let $Y$ be a $3$-dimensional  Alexandrov space. 
\begin{itemize}
	\item[(1)]If $\curv (Y)\geq 1$, then $Y$ has at most three points for which the space of directions is not larger than $\sphere^2(1/2)$.\smallskip
	\item[(2)]If $\curv (Y)\geq 0$, then $Y$ has at most four points for which the space of directions is not larger than $\sphere^2(1/2)$.\smallskip
\end{itemize}
\end{lem}


\subsection{Topological classification}


We first consider the case where $\curv(X)\geq 1$.
Suppose that the fixed point set $F$ contains a $2$-sphere, which must be a boundary component of the orbit space $X^*$. Since $X^*$ is a positively curved Alexandrov space, the distance function to $\partial X^*$ is strictly concave and it follows from Perelman's Soul Theorem for Alexandrov spaces \cite{Pe} that the set at maximal distance from $\partial X^*$ is a point $x_0^*$.  The distance function to $\partial X^*$  has no critical points in 
$X^*- \{\, \partial X^*\cup \{\,x_0^*\,\}\,\}$  and, proceeding as in \cite[Lemma~1.1-(iii)]{GS}, one concludes that the points in $X^*- \{\, \partial X^*\cup \{\,x_0^*\,\}\,\}$ correspond to orbits with trivial isotropy.  Hence, the fixed point set of the action is either a $2$-sphere, or a $2$-sphere and an isolated fixed point. If $F$ consists only of  isolated fixed points, then  there are at most three such points,  by Lemma~\ref{L:FP_Bound}. In all cases, we conclude that $\chi(F)\leq 3$. 
  

Suppose now that $\curv(X)\geq 0$. We have three cases to consider: first, when the fixed point set $F$ contains no $2$-dimensional component; second, when $F$ contains exactly one $2$-dimensional component; third, when $F$ contains two or more $2$-dimensional components. 

In the first case, $F$ consists only of isolated fixed points and, by Lemma~\ref{L:FP_Bound}, there can be at most four such points. 

In the second case, we must show that there are at most two isolated fixed points.  Observe that the orbit space $X^*$ is a $3$-ball with an Alexandrov metric of nonnegative curvature. By Perelman's Doubling Theorem \cite{Pe}, the double $DX^*$ of $X^*$ is a $3$-sphere with an Alexandrov metric of nonnegative curvature. By Lemma~\ref{L:FP_Bound}, $DX^*$ contains at most four points  for which the space of directions is not larger than $\sphere^2(1/2)$. Therefore, $X^*$ can contain at most two isolated fixed points. 

In the third case, there are at least two $2$-dimensional fixed point set components, $F_1^*$ and $F^*_2$, and they are boundary components of $X^*$. Let $C^*$ be the set at maximal distance from $F^*_1$. Since $X^*$ is  nonnegatively curved, the distance function to $F^*_1$ has no critical points in $X^*-\{F^*_1\cup C^*\}$. Therefore, $F^*_2\subset C^*$ and we conclude that $X^*$ has exactly two boundary components and there cannot be isolated fixed points. Alternatively, since each $2$-dimensional fixed point set component is a primitive extremal subset of $X^*$, it follows from W\"orner's Splitting Theorem \cite{Wo} that $X^*$ must split isometrically as $F_1^*\times I$, where $I$ is a closed interval. 

We conclude that in the nonnegatively curved case  $F$ consists of four isolated fixed points,  a $2$-sphere and at most two isolated fixed points, or two $2$-spheres. Hence, $\chi(F)\leq 4$.


\subsection{Equivariant classification}


Let $X$ be isometric to one of $\sphere^4$, $\CP^2$, $\sphere^2\times\sphere^2$ or $\CP^2\#\pm\CP^2$ equipped with an Alexandrov metric of nonnegative curvature and an isometric circle action.  By work of Orlik and Raymond \cite{OR}, any locally smooth smooth $\T^2$ action on $\sphere^4$, $\CP^2$ or  
$\sphere^2\times\sphere^2$ or $\CP^2\#\pm\CP^2$  is equivalent to a smooth action and any smooth $\T^2$ action on $\sphere^4$ or $\CP^2$  is equivalent to a linear action. On the other hand, by \cite{GGK}, any smooth $\T^2$ action on $\sphere^2\times\sphere^2$ or $\CP^2\#\pm\CP^2$ is equivalent to an isometric action on a normal biquotient. Thus, to obtain the equivariant classification of locally smooth circle actions on $X$ it suffices to show that the circle action extends to a locally smooth $\T^2$ action. By the discussion in Subsection~\ref{SS:SETUP}, it is enough to verify that if $K^*\subset \overline{E}^*$ is a simple closed curve in $X^*$, then $K^*=E^*\cup F^*$ and $K^*$ is unknotted in $X^*\simeq \sphere^3$. To do so, we apply the techniques developed by Grove and Wilking in \cite{GW}. The following result is pivotal.


\begin{lem}[Lemma~5.2, \cite{GW}]
\label{L:DBL_BC}  The double branched covering of $\sphere^3$ with an Alexandrov metric of nonnegative curvature and branching locus an extremal simple closed curve is an Alexandrov space of nonnegative curvature. 
\end{lem}

 By the following proposition, we may assume that $F^*$ consists only of isolated fixed points, so that $X^*\simeq \sphere^3$.


\begin{prop} 
\label{P:FPH_NOCIRCLES}
If $F^*$ contains a $2$-dimensional component, then $E^*\cup F^*$ does not contain a simple closed curve.
\end{prop}


\begin{proof}
When $\chi(F^*)=2$ or $3$, the set $E^*$ is empty. Therefore, in these two cases $E^*\cup F^*=F^*$ cannot contain a simple closed curve. When $\chi(F^*)=4$, the fixed point set $F^*$ consists of two $2$-spheres, or a $2$-sphere and two isolated fixed points. In the first case, $E^*$ must be empty and $E^*\cup F^*=F^*$ does not contain a simple closed curve. Suppose now that $F^*$ is the union of a $2$-sphere and two isolated fixed points, so that $X^*$ is a $3$-ball with an Alexandrov metric of nonnegative curvature. 

Assume, for the sake of contradiction, that $K^*\subset E^*\cup F^*$ is a simple closed curve; the curve $K^*$ must contain the two isolated fixed points in $F^*$. Since $X^*$ is a $3$-ball with an Alexandrov metric of nonnegative curvature, by Perelman's Doubling Theorem \cite{Pe}, the double $DX^*$ of $X^*$ is a $3$-sphere with an Alexandrov metric of nonnegative curvature. There are two copies of $K^*$ in $DX^*$, each one with two points whose space of directions is not larger than $\sphere^2(1/2)$. Since $K^*$ is an extremal simple closed curve in $DX^*\simeq \sphere^3$,  the double branched cover $DX^*_2$ with branching locus $K^*$ is a $3$-dimensional Alexandrov space of nonnegative curvature, by Lemma~\ref{L:DBL_BC}. Moreover, the space of directions at points in the branching locus of $DX^*_2$ corresponding to isolated fixed points in $K^*$ remains not larger than $\sphere^2(1/2)$ (cf.~\cite[Section~5]{GW}). Therefore, $DX^*_2$ has six points whose space of directions is  not larger than $\sphere^2(1/2)$, which contradicts Lemma~\ref{L:FP_Bound}. 
 \end{proof}

Let $K^*\subset \overline{E}^*$ be a simple closed curve in $X^*\simeq \sphere^3$. To  conclude the proof, we first  verify that $K^*=E^*\cup F^*$. If $F^*$ consists of two isolated fixed points, then $K^*=E^*\cup F^*$ and there is nothing to prove in this case. If $F^*$ consists of three isolated fixed points, it follows from  \cite[Lemma 5.1]{F1} that $E^*\cup F^*$ cannot be the disjoint union of $K^*$ and an isolated fixed point. Suppose now that $F^*$ consists of four isolated fixed points and at least one of them is not contained in $K^*$. A double branched cover argument as in the proof of Proposition~\ref{P:FPH_NOCIRCLES} yields a contradiction in this case. Hence, $K^*=E^*\cup F^*$ in all cases.

Finally, to verify that $K^*$ is the unknot, let $X^*_2$ be the double branched cover of $X^*$ with branching locus $K^*$. If $K^*$ is knotted, then $\pi_1(X^*_2)$ has order at least three \cite[Theorem~C]{GW}. On the other hand, $X^*_2$ is a $3$-dimensional Alexandrov space of nonnegative curvature with at least two points whose space of directions is not larger than $\sphere^2(1/2)$.  It follows that, if $K^*$ is knotted, the universal cover of $X^*_2$ is a $3$-dimensional Alexandrov space with at least six points with space of directions not larger than $\sphere^2(1/2)$, which contradicts  Lemma~\ref{L:FP_Bound}.

\bibliographystyle{amsplain}


\end{document}